\documentclass[12pt]{amsart}

\usepackage{amssymb, amscd, txfonts}
\usepackage{graphicx}

\numberwithin{equation}{section}

\newtheorem{theorem}{Theorem}[section]
\newtheorem{proposition}[theorem]{Proposition}

\theoremstyle{definition}

\theoremstyle{remark}

\newcommand{\Z}{\mathbb{Z}}

\newcommand{\Q}{\mathbb{Q}}
\newcommand{\R}{\mathbb{R}}
\newcommand{\C}{\mathbb{C}}
\newcommand{\HH}{\mathbb{H}}
\newcommand{\proj}{{\mathbb P}}

\newcommand{\SL}{{\rm SL}_2(\mathbb{Z})}

\begin{document}

%%%%%%% Title %%%%%%%%%%%%%%%%%%%%%%%%
\title[]{Modular forms of weight $3m$ and elliptic modular surface}
\author[]{Shouhei Ma}
\thanks{Supported by Grant-in-Aid for Scientific Research 15H02051} 
\address{Department~of~Mathematics, Tokyo~Institute~of~Technology, Tokyo 152-8551, Japan}
\email{ma@math.titech.ac.jp}
%\subjclass[2000]{11F55, 17B67, 11F22, 14J28} AG, NT, RT, QA 
%\dedicatory{}

\begin{abstract}
We prove that the graded ring of modular forms of weight divisible by $3$ is naturally isomorphic to 
a certain log canonical ring of the corresponding elliptic modular surface. 
\end{abstract} 

\maketitle

%%%%%%%%Main result%%%%%%%%%%%

%\section{Main result}

Let $\Gamma$ be a subgroup of ${\SL}$ of finite-index which does not contain $-1$. 
In  \cite{Sh}, Shioda introduced the elliptic modular surface 
$\pi: S_{\Gamma}\to X_{\Gamma}$ over the (compactified) modular curve $X_{\Gamma}$ associated to $\Gamma$, 
and proved, among other things, that the space of cusp forms of weight $3$ with respect to $\Gamma$ is canonically isomorphic to 
that of holomorphic $2$-forms on the surface $S_{\Gamma}$. 
Our purpose is to extend this correspondence to that between 
modular forms of weight $3m$ and certain rational $m$-pluricanonical forms on $S_{\Gamma}$. 
This is parallel to the situation for weight $2m$ 
where modular forms of weight $2m$ correspond to $m$-pluricanonical forms on the curve $X_{\Gamma}$. 

To state the result, let $\Delta\subset X_{\Gamma}$ be the (reduced) cusp divisor and consider the decomposition 
$\Delta=\Delta_{reg}+\Delta_{irr}$ into regular and irregular cusps. 
The singular fibres over $\Delta$ are divided accordingly, which we denote as 
$\pi^{\ast}\Delta=D_{reg}+D_{irr}$. 
The divisor $D_{reg}$ consists of type $I_n$ fibres, while $D_{irr}$ consists of type $I_{n}^{\ast}$ fibres ($n$ depends on the cusps). 
We consider the ${\Q}$-divisor 
\begin{equation*}
D = D_{reg} + \frac{1}{2}D_{irr}. 
\end{equation*}
Write $M_{k}(\Gamma)$ for the space of $\Gamma$-modular forms of weight $k$.

\begin{theorem}\label{main}
We have a natural isomorphism of graded rings 
\begin{equation}\label{eqn:main}
\bigoplus_{m\geq0} M_{3m}(\Gamma) \: \:  \simeq \: \: \bigoplus_{m\geq0} H^0(K_{S_{\Gamma}}^{\otimes m}(mD)). 
\end{equation}
\end{theorem}

Here $K_{S}^{\otimes m}(mD)$ should be understood as $K_{S}^{\otimes m}(mD_{reg}+[m/2]D_{irr})$ as usual. 
Explicitly, the isomorphism \eqref{eqn:main} is given by associating to $f(\tau)\in M_{3m}(\Gamma)$ the $m$-canonical form 
$f(\tau)(d\tau\wedge dz)^{\otimes m}$ where $z$ is the natural coordinate on the smooth fibres of $\pi$. 
Determination of the precise pole condition is the main content of this paper. 

We see some contrasts with the analogous isomorphism in weight $2m$, which can be written in the form  
\begin{equation*}
\bigoplus_{m\geq0} M_{2m}(\Gamma) \:  \simeq \:  \bigoplus_{m\geq0} H^0(K_{X_{\Gamma}}^{\otimes m}(m(\Delta+(2/3)B))), 
\end{equation*}
where $B$ is the divisor of elliptic points (of order $3$). 
The first is that the singular fibres over elliptic points do not contribute to the pole condition, 
and the second is the difference between the coefficients at regular and irregular cusps. 

Independently of the proof, we can check that 
the $m$-th components of both sides of \eqref{eqn:main} indeed have the same dimension, which is given by  
\begin{equation*}
(3m-1)(g-1) + m\varepsilon_{3} + \frac{3m}{2}\varepsilon_{reg} + \left [\frac{3m}{2} \right] \varepsilon_{irr}. 
\end{equation*}
Here $g$ is the genus of $X_{\Gamma}$, and 
$\varepsilon_{3}, \varepsilon_{reg}, \varepsilon_{irr}$ are the number of elliptic points, regular cusps and irregular cusps respectively. 
The dimension formula for $M_{k}(\Gamma)$ is well-known, 
while the dimension of $H^0(K_{S_{\Gamma}}^{\otimes m}(mD))$ can be calculated using the canonical bundle formula and 
Riemann-Roch on $X_{\Gamma}$. 

By transposing $mD$ in \eqref{eqn:main}, 
we also obtain an isomorphism between the canonical ring of $S_{\Gamma}$ 
and the following subring of $\oplus_{m} M_{3m}(\Gamma)$: 
\begin{equation*}
\bigoplus_{m\geq0}H^0(K_{S_{\Gamma}}^{\otimes m}) 
\: \: \simeq \: \: 
\bigoplus_{m\geq0} \bigl\{ f\in M_{3m}(\Gamma) \; | \; {\rm ord}_s(f)\geq m \: \textrm{at every cusp} \: s \bigr\}. 
\end{equation*}
Here the vanishing order of $f$ at a cusp $s=\gamma(i\infty)$, $\gamma\in{\SL}$, is measured by the parameter $e^{2\pi i\tau /N}$ 
where $N$ is the smallest positive integer such that 
$\begin{pmatrix} 1 & N \\ 0 & 1 \end{pmatrix}$ is contained in $\gamma^{-1} \Gamma \gamma$. 
When $m=1$, this is the Shioda isomorphism $H^0(K_{S_{\Gamma}}) \simeq S_3(\Gamma)$. 

%One might wonder that the effect of elliptic points is seemingly not appearing in the right hand side of \eqref{eqn:main}; 
%in fact, it has already affected on the discrepancy between $(\pi_{\ast}K_{\pi})^{\otimes k}$ and the line bundle of 
%modular forms of weight $k$. 

%Shioda's isomorphism has been generalized in various directions, 
% to higher dimensional Kuga varieties ([Shokurov]) and also to other modular groups ([Lee]). 
%It could be a more complicated task obtain a similar extension in those cases. 

The proof of Theorem \ref{main} is given in \S \ref{sec:proof}. 
In \S \ref{sec:interpret} we explain the interpretation of the Hecke operators on $M_{3m}(\Gamma)$ 
and the Petersson inner product on $S_3(\Gamma)$ in terms of the pluricanonical forms.

%%%%%%%%Proof%%%%%%%%%%%%%%

\section{Proof}\label{sec:proof}

Before proceeding to the proof, let us first explain the source of \eqref{eqn:main} at the level of period domain. 
We prefer to view the upper half plane as the open set 
\begin{eqnarray*}
\mathcal{D} & = &  \{ [\omega] \in {\proj}^1 \: | \: \sqrt{-1}(\omega, \bar{\omega})>0 \} \\
& = & \{ [ \alpha, \beta ] \in {\proj}^1 \: | \: {\rm Im}(\alpha/\beta)>0 \} 
\end{eqnarray*}
of ${\proj}^1$, 
where $( \, , \, )$ is the anti-standard symplectic form on ${\C}^2$.  
The restricted tautological bundle $\mathcal{L}=\mathcal{O}_{{\proj}^1}(-1)|_{\mathcal{D}}$, 
with the natural ${\rm GL}_2^+({\R})$-action, 
is the line bundle of modular forms of weight $1$. 
Let $\underline{{\C}^2}={\C}^2\times\mathcal{D}$ be the product vector bundle 
and let $\underline{{\Z}^2}={\Z}^2\times\mathcal{D}$, viewed as a system of sections of $\underline{{\C}^2}$. 
The universal marked elliptic curve $\tilde{\pi}:\mathcal{S}\to\mathcal{D}$ over $\mathcal{D}$ is the quotient 
\begin{equation*}
\mathcal{S} = \underline{{\C}^2}/\mathcal{L}+\underline{{\Z}^2} \simeq \mathcal{L}^{\vee}/\underline{{\Z}^2}^{\vee}. 
\end{equation*}
The second isomorphism is induced by the symplectic form. 
The local system $R^1\tilde{\pi}_{\ast}{\Z}$ is identified with $\underline{{\Z}^2}$, 
and the Hodge bundle of $\mathcal{S}\to\mathcal{D}$ is identified with $\mathcal{L}$. 
Let $K_{\tilde{\pi}}$ be the relative canonical bundle of $\mathcal{S}/\mathcal{D}$.  
Since $K_{\tilde{\pi}}|_F\simeq K_F$ is trivial at each fiber $F$, 
$\tilde{\pi}_{\ast}K_{\tilde{\pi}}$ is an invertible sheaf on $\mathcal{D}$, and 
the natural homomorphism 
$(\tilde{\pi}_{\ast}K_{\tilde{\pi}})^{\otimes m}\to \tilde{\pi}_{\ast}(K_{\tilde{\pi}}^{\otimes m})$ 
is isomorphic for any $m$. 

We have two fundamental ${\SL}$-equivariant isomorphisms: 
\begin{equation}\label{eqn:isom at domain}
K_{\mathcal{D}} \simeq \mathcal{L}^{\otimes2}, \qquad \tilde{\pi}_{\ast}K_{\tilde{\pi}} \simeq \mathcal{L}. 
\end{equation}
The first one is just 
$K_{\mathcal{D}}=K_{{\proj}^1}|_{\mathcal{D}}\simeq \mathcal{O}_{{\proj}^1}(-2)|_{\mathcal{D}}$, 
and the second is given by integrating the relative $1$-forms. 
Combining these isomorphisms, we obtain  
\begin{equation*}\label{eqn:isom at domain II}
\mathcal{L}^{\otimes 3m} 
\simeq (\tilde{\pi}_{\ast}K_{\tilde{\pi}})^{\otimes m} \otimes K_{\mathcal{D}}^{\otimes m} 
\simeq \tilde{\pi}_{\ast}(K_{\tilde{\pi}}^{\otimes m} \otimes \tilde{\pi}^{\ast} K_{\mathcal{D}}^{\otimes m}) 
\simeq \tilde{\pi}_{\ast}(K_{\mathcal{S}}^{\otimes m}). 
\end{equation*}
This is the source of \eqref{eqn:main}. 

Now let $\Gamma < {\SL}$ be a given group. 
Taking quotient of $\mathcal{S}\to\mathcal{D}$ by $\Gamma$ and resolving the $A_2$-singularities arising from the fixed points, 
we obtain an elliptic fibration over $Y_{\Gamma}=\Gamma\backslash\mathcal{D}=X_{\Gamma}-\Delta$. 
The elliptic modular surface $\pi:S_{\Gamma}\to X_{\Gamma}$ 
is the non-singular, relatively minimal extension of this fibration over $X_{\Gamma}$. 
Let us abbreviate $Y_{\Gamma}, X_{\Gamma}, S_{\Gamma}$ as $Y, X, S$ respectively.  
We shall prove Theorem \ref{main} in two steps.

\subsection{Semi-stable case}\label{ssec:semistable}

We first consider the case where $\Gamma$ has no elliptic point nor irregular cusp. 
In particular, $\Gamma$ acts on $\mathcal{D}$ freely and $\pi$ has only singular fibres of type $I_n$. 
For instance, $\Gamma(N)$ with $N>2$ and more generally neat subgroups of ${\SL}$ satisfy this condition. 

The $\Gamma$-equivariant bundles $K_{\mathcal{D}}$, $\tilde{{\pi}}_{\ast}K_{\tilde{\pi}}$ and $\mathcal{L}$ 
descend to line bundles on $Y$. 
We extend them over $X$ as follows. 
The first one, $K_Y$, is extended to $K_X$. 
The second one is extended to $\pi_{\ast}K_{\pi}$ where $K_{\pi}$ is the relative canonical bundle of $\pi$. 
Recall that (local) sections of $K_{\pi}|_F$ at a singular fibre $F$ are identified with 
(local) $1$-forms on $F\backslash{\rm Sing}(F)$ such that 
at each node it has pole of order $\leq1$ and the sum of its residues over the two branches equals to $0$.  
Hence $K_{\pi}|_F\simeq \mathcal{O}_F$, so that $\pi_{\ast}K_{\pi}$ is still invertible and we have 
$(\pi_{\ast}K_{\pi})^{\otimes m}\simeq \pi_{\ast}(K_{\pi}^{\otimes m})$. 

The third one, the descent of $\mathcal{L}$, is extended as follows. 
Let $l$ be a primitive vector of ${\Z}^2$, which corresponds to the rational boundary point $[l]\in{\proj}^1$ of $\mathcal{D}$. 
Let $s_{l}$ be the rational section of $\mathcal{O}_{{\proj}^1}(-1)$ 
defined by the condition $(s_l([\omega]), l)=1$, which has pole of order $1$ at $[l]$. 
Since $s_l$ is invariant under the stabilizer of $l$, 
then $s_{l}|_{\mathcal{D}}$ induces a local frame of the descent of $\mathcal{L}$ near the corresponding cusp. 
The extension over the cusp is defined so that this punctured local frame extends as a local frame. 
We write $L$ for the extended line bundle over $X$. 
A local $\Gamma$-invariant section $s$ of $\mathcal{L}$ 
extends holomorphically over the cusps as a local section of $L$ 
if and only if $(s([\omega]), l)$ does not diverge as $[\omega]\to [l]$. 
This coincides with the usual cusp condition for modular forms. 
In particular, we have $M_k(\Gamma)=H^0(L^{\otimes k})$. 

Now the isomorphisms \eqref{eqn:isom at domain} descend to 
\begin{equation*}
K_Y \simeq L^{\otimes 2}|_Y, \qquad \pi_{\ast}K_{\pi}|_Y\simeq L|_Y. 
\end{equation*}
As is well-known, the first one extends to $K_X \simeq L^{\otimes 2}(-\Delta)$. 
The second one extends to $\pi_{\ast}K_{\pi}\simeq L$. 
Indeed, in the period map around a singular fiber $F$, 
the vanishing cycle near a node of $F$ corresponds to the cusp vector $l\in{\Z}^2$. 
The integral of a generator of $H^0(K_{\pi}|_F)$ along the vanishing cycle is equal to its residue at the node, whence nonzero. 
Therefore a local frame of $\pi_{\ast}K_{\pi}$ corresponds to that of $L$ under the period integral isomorphism 
$\pi_{\ast}K_{\pi}|_Y\simeq L|_Y$. 

Then we obtain 
\begin{equation*}
L^{\otimes 3m} 
\simeq (\pi_{\ast}K_{\pi})^{\otimes m}  \otimes  K_X^{\otimes m}(m\Delta) 
\simeq \pi_{\ast}(K_{\pi}^{\otimes m} \otimes \pi^{\ast}K_X^{\otimes m}(mD))
\simeq \pi_{\ast}(K_S^{\otimes m}(mD)).  
\end{equation*}
Taking global section gives \eqref{eqn:main}. 
Compatibility of the multiplications is obvious.

\subsection{General case}

We next study the general case. 
Let $\Gamma<{\SL}$ be a given group. 
We choose as an auxiliary step a normal subgroup $\Gamma'\lhd\Gamma$ of finite-index 
that has no elliptic point nor irregular cusp. 
We will abbreviate the elliptic modular surfaces $S_{\Gamma}\to X_{\Gamma}$ and $S_{\Gamma'}\to X_{\Gamma'}$ as 
$\pi\colon S\to X$ and $\pi'\colon S'\to X'$ respectively. 
The quotient group $\bar{\Gamma}=\Gamma/\Gamma'$ acts on $S'$ biregularly.  
$S'/X'$ is the nonsingular, relatively minimal elliptic surface birational to the base change $X'\mathop{\times}_{X} S$ 
of $S/X$ by the natural projection $f\colon X'\to X$. 
Let us observe this process of birational transformation around each singular fibre of $\pi$, 
with emphasis on the relation of canonical divisors. 

Let $F=\pi^{\ast}p$ be a singular fibre of $\pi$. 
Choose an arbitrary point $p'\in f^{-1}(p)$ and let $F'=(\pi')^{\ast}p'$ be the fibre over $p'$. 
Take a suitable small neighborhood $V\subset X$ of $p$ and let $V'\subset X'$ be the component of $f^{-1}(V)$ that contains $p'$. 
We write $U=\pi^{-1}(V)$ and $U'=(\pi')^{-1}(V')$. 
The stabilizer $G\subset\bar{\Gamma}$ of $p'$ is cyclic, say of order $d$, and we have $V\simeq V'/G$ naturally. 
%Note that $G$ acts on $U'$ biregularly. 

(1) 
When $p$ is a regular cusp, $F$ is a type $I_n$ fibre. 
The fibre product $U''=V'\mathop{\times}_{V}U$ is normal and 
has $A_{d-1}$-singularities at the $n$ nodes of the central fibre $F''$ of $U''\to V'$ (which is identified with $F$). 
Then $U'$ is the minimal resolution of those $A_{d-1}$-points, and thus $F'$ is of type $I_{dn}$. 
If $p:U'\to U''$ and $f:U''\to U$ are the natural maps, then 
\begin{equation*}
K_{U'}(F') \simeq p^{\ast}(K_{U''}(F'')) \simeq (f\circ p)^{\ast}(K_U(F)). 
\end{equation*}
The first isomorphism holds because $A_{d-1}$-points are canonical singularity and $p^{\ast}F''=F'$, 
and the second one is the (log) ramification formula. 
In particular, we have 
\begin{equation}\label{eqn:local model1}
H^0(K_{U'}^{\otimes m}(mF'))^{G} = H^0(K_{U''}^{\otimes m}(mF''))^{G} = H^0(K_{U}^{\otimes m}(mF)). 
\end{equation}

(2) 
When $p$ is an irregular cusp, $F$ is of type $I_{n}^{\ast}$, and $d$ is an even number, say $d=2d'$. 
Let $G'$ be the subgroup of $G$ of order $d'$ and let 
$G''=G/G'\simeq{\Z}/2$. 
We set $V''=V'/G'$ and look at the factorization 
%\begin{equation*}
$V' \to V'' \to V$. 
%\end{equation*}
Let $p\colon U\to \bar{U}$ be the contraction of the four components of $F$ of multiplicity $1$. 
The central fibre of $\bar{U}\to V$ is a multiple fibre of multiplicity $2$, say $2\bar{F}$. 
Let $U''$ be the normalization of the fibre product $V''\mathop{\times}_{V}\bar{U}$; 
this is the double cover of $\bar{U}$ branched over the four $A_1$-points, and in particular smooth. 
The central fibre $F''$ of $U''\to V''$ is of type $I_{2n}$. 
Thus $U''/V''$ is the nonsingular, relatively minimal model of $V''\mathop{\times}_{V}U$. 
Since the covering map $f''\colon U''\to\bar{U}$ is unramified outside the $A_1$-points, we have 
\begin{equation*}
K_{U''}+F'' = (f'')^{\ast}(K_{\bar{U}}+\bar{F}), \qquad 
p^{\ast}(K_{\bar{U}}+\bar{F}) = K_{U} + \frac{1}{2}F. 
\end{equation*}
Here $p^{\ast}\bar{F}$ means pullback of ${\Q}$-Cartier divisor. 
Therefore 
\begin{equation*}
H^0(K_{U''}^{\otimes m}(mF''))^{G''} = H^0(K_{\bar{U}}^{\otimes m}(m\bar{F})) = H^0(K_{U}^{\otimes m}((m/2)F)). 
\end{equation*}
On the other hand, the relation of $U'/V'$ and $U''/V''$ is the same as described in the case (1). 
Hence we have a natural isomorphism 
\begin{equation}\label{eqn:local model2}
H^0(K_{U'}^{\otimes m}(mF'))^{G} = H^0(K_{U}^{\otimes m}((m/2)F)). 
\end{equation}

(3) 
When $p$ is an elliptic point (of order $3$), $F$ is of type $IV^*$, $F'$ is smooth (of $j$-invariant $0$), and $d=3$. 
In this case it is more convenient to look from $F'$. 
The covering transformation group $G\simeq{\Z}/3$ acts on $U'$ with three isolated fixed points at $F'$. 
The quotient $U''=U'/G$ has three $A_2$-singularities correspondingly 
and the central fibre of $U''\to V$ has multiplicity $3$. 
Resolving those $A_2$-points, we obtain $U$. 
If $f\colon U'\to U''$ and $p\colon U\to U''$ are the natural maps, we have 
\begin{equation*}
K_{U'}\simeq f^{\ast}K_{U''}, \qquad p^{\ast}K_{U''}\simeq K_{U}
\end{equation*}
as before. 
Hence 
\begin{equation}\label{eqn:local model3}
H^0(K_{U'}^{\otimes m})^{G} = H^0(K_{U''}^{\otimes m}) = H^0(K_{U}^{\otimes m}). 
\end{equation}

Now we can complete the proof of Theorem \ref{main}. 
Let $D'$ be the sum of the singular fibres of $\pi'$. 
By \S \ref{ssec:semistable} we have an isomorphism 
\begin{equation*}
\bigoplus_{m} M_{3m}(\Gamma') \simeq \bigoplus_{m} H^0(K_{S'}^{\otimes m}(mD')), 
\end{equation*}
which is $\bar{\Gamma}$-equivariant by construction. 
We take the $\bar{\Gamma}$-invariant part. 
On the one hand, we have 
\begin{equation*}
M_{3m}(\Gamma')^{\bar{\Gamma}} = M_{3m}(\Gamma). 
\end{equation*}
On the other hand, by the local analysis \eqref{eqn:local model1}, \eqref{eqn:local model2}, \eqref{eqn:local model3}, 
we see that 
\begin{equation*}
H^0(K_{S'}^{\otimes m}(mD'))^{\bar{\Gamma}} \simeq H^0(K_{S}^{\otimes m}(mD)). 
\end{equation*}
This gives \eqref{eqn:main}.

\section{Some interpretation}\label{sec:interpret}

The Hecke operators on $M_{3m}(\Gamma)$ and the Petersson scalar product on $S_3(\Gamma)$ 
have natural interpretation in terms of the pluricanonical forms. 
The weight $3$ case must be well-known, but included here because 
we could not find suitable reference. 

\subsection{Hecke operators}

Assume that $\Gamma$ is a congruence subgroup, and let $\alpha$ be an element of $M_2^+({\Z})$. %with ${\det}(\alpha)>0$. 
Putting $\Gamma_{\alpha}=\Gamma\cap \alpha\Gamma\alpha^{-1}$ and $\Gamma_{\alpha}=\Gamma\cap \alpha^{-1}\Gamma\alpha$, 
we have the self correspondence 
\begin{equation}\label{eqn:curve correspo}
X_{\Gamma} \stackrel{\pi_1}{\gets} X_{\Gamma_{\alpha}} \stackrel{\alpha}{\gets} X_{\Gamma^{\alpha}} \stackrel{\pi_2}{\to} X_{\Gamma}
\end{equation}
of $X_{\Gamma}$, where $\pi_i$ are the projections and 
$\alpha$ is the isomorphism induced from the $\alpha$-action on $\mathcal{D}$. 
This induces the Hecke operator 
\begin{equation*}
[\Gamma\alpha\Gamma]_k := {\det}(\alpha)^{k-1}\pi_{2\ast}\circ\alpha^{\ast}\circ\pi_{1}^{\ast}
\end{equation*}
on $M_k(\Gamma)$. 
Here $\alpha^{\ast}$ is induced from the \textit{natural} $\alpha$-action on $\mathcal{L}^{\otimes k}$. 

On the other hand, the curve correspondence \eqref{eqn:curve correspo} lifts to 
the rational correspondence of the elliptic modular surfaces 
\begin{equation}\label{eqn:surface correspo}
S_{\Gamma} \stackrel{\hat{\pi_1}}{\dashleftarrow} S_{\Gamma_{\alpha}} 
\stackrel{\hat{\alpha}}{\dashleftarrow} S_{\Gamma^{\alpha}} \stackrel{\hat{\pi_2}}{\dashrightarrow} S_{\Gamma}. 
\end{equation}
Here $\hat{\pi_i}$ are the base change maps 
%\begin{equation*}
%S_{\Gamma'} \sim X_{\Gamma'}\times_{X_{\Gamma}}S_{\Gamma} \to S_{\Gamma}
%\end{equation*}
and isomorphic at each smooth fiber; 
$\hat{\alpha}$ is induced from the natural $\alpha$-action on the bundle $\underline{{\C}^2}$ 
and gives an isogeny of degree ${\det}(\alpha)$ at each smooth fiber. 
The indeterminacy points and ramification divisors of $\hat{\alpha}, \hat{\pi_1}, \hat{\pi_2}$ are 
contained in the singular fibers and their preimage. 
The surface correspondence \eqref{eqn:surface correspo} induces the endomorphism 
$\hat{\pi_{2}}_{\ast}\circ\hat{\alpha}^{\ast}\circ\hat{\pi_{1}}^{\ast}$ 
on rational pluricanonical forms on $S_{\Gamma}$. 

\begin{proposition}
Through the isomorphism \eqref{eqn:main} the modular Hecke operator $[\Gamma\alpha\Gamma]_{3m}$ 
agrees with the endomorphism 
${\det}(\alpha)^{2m-1}\hat{\pi_{2}}_{\ast}\circ\hat{\alpha}^{\ast}\circ\hat{\pi_{1}}^{\ast}$
on $H^0(K_{S_{\Gamma}}^{\otimes m}(mD))$. 
\end{proposition}

\begin{proof}
It is clear that the pullbacks $\pi_1^{\ast}, \hat{\pi_1}^{\ast}$ agree 
and that the trace maps $\pi_{2\ast}, \hat{\pi_2}_{\ast}$ agree. 
We verify that the $m$-canonical form pullback $\hat{\alpha}^{\ast}$ 
corresponds to ${\det}(\alpha)^m$ multiple of the weight $3m$ modular form pullback $\alpha^{\ast}$. 
Since both operators are locally defined and compatible with multiplication, 
we only have to check this for $m=1$ and at the level of period domain $\mathcal{D}$. 
Under the isomorphism $\mathcal{L}^{\otimes 2}\simeq K_{\mathcal{D}}$, 
the weight $2$ modular pullback agrees with the pullback of $1$-forms on $\mathcal{D}$. 
On the other hand, under $\mathcal{L}\simeq \tilde{\pi}_{\ast}K_{\tilde{\pi}}$, 
the weight $1$ modular pullback agrees with ${\det}(\alpha)^{-1}$ multiple of the $1$-form pullback on the fibers 
by the isogeny $\alpha_{[\omega]}:\mathcal{S}_{[\omega]}\to\mathcal{S}_{[\alpha\omega]}$, 
because $\alpha$ multiplies the symplectic form by ${\det}(\alpha)$. 
\end{proof}

\subsection{Petersson inner product in weight $3$}

Let $( \, , \, )_{\mathcal{L}}$ be the ${\rm SL}_2({\R})$-invariant Hermitian metric on $\mathcal{L}$ 
that corresponds to (half) the Hodge metric 
\begin{equation*}
(\omega_1, \omega_2) = \frac{\sqrt{-1}}{2} \int_E \omega_1\wedge\bar{\omega_2}, \qquad \omega_i \in H^0(K_E)
\end{equation*}
in each fiber $\mathcal{L}_{[\omega]}=H^0(K_E)$ where $E=\mathcal{S}_{[\omega]}$. 
Let $2\Omega$ be the K\"ahler form of the metric induced on $\mathcal{L}^{\otimes-2}\simeq T_{\mathcal{D}}$. 
In the upper half plane model $\Omega$ is expressed as $y^{-2}dx\wedge dy$. 
The Petersson inner product on $S_3(\Gamma)$ is defined by 
\begin{equation*}
(f, g) = \int_{Y_{\Gamma}} (f, g)_{\mathcal{L}^{\otimes3}}\Omega, \qquad f, g\in S_3(\Gamma). 
\end{equation*}
Via the trivialization of $\mathcal{L}$ given by the frame 
${\HH}\ni\tau\mapsto(\tau, 1)\in\mathcal{L}$, 
%$s_{(1,0)}=\{ (\tau, 1)\in\mathcal{L}^{\times} | \tau\in{\mathbb H} \}$, 
this agrees with the classical form 
\begin{equation*}
\int_{Y_{\Gamma}} f(\tau)\overline{g(\tau)}y dx\wedge dy. 
\end{equation*}

\begin{proposition}
Through the Shioda isomorphism $S_3(\Gamma)\simeq H^0(K_{S_{\Gamma}})$ 
the Petersson inner product agrees with the ($1/4$-scaled) Hodge metric on $H^0(K_{S_{\Gamma}})$ 
\begin{equation}\label{eqn:Hodge metric}
(\omega, \eta) = \frac{1}{4} \int_{S_{\Gamma}} \omega\wedge\bar{\eta}, \qquad \omega, \eta \in H^0(K_{S_{\Gamma}}). 
\end{equation}
\end{proposition}

\begin{proof}
We have $(\eta_1, \eta_2)_{\mathcal{L}^{\otimes2}}2\Omega=\sqrt{-1}\eta_1\wedge\bar{\eta_2}$ 
for $1$-forms $\eta_i$ on $\mathcal{D}$. 
If $f, g\in S_3(\Gamma)$ correspond to $2$-forms which locally can be written in the form 
$\tilde{\pi}^{\ast}\eta_1\otimes \omega_1$ and $\tilde{\pi}^{\ast}\eta_2\otimes \omega_2$, 
we locally have the equality of $(1, 1)$-forms 
\begin{equation*}
(f, g)_{\mathcal{L}^{\otimes3}}\Omega = 
-\frac{1}{4} \left( \int_E\omega_1\wedge\bar{\omega_2} \right) \eta_1\wedge\bar{\eta_2}. 
\end{equation*}
Integration of the right hand side over $Y_{\Gamma}$ gives \eqref{eqn:Hodge metric}. 
\end{proof}

%%%%%%% Reference %%%%%%%%%%%%%%%%%%%%%%%%%%%%%

\end{document}